\documentclass[11pt]{amsart}

\usepackage[margin=1.25in]{geometry} 
\usepackage[table]{xcolor}
\usepackage{amsmath,amsthm,amssymb,mathtools, scrextend,enumerate,bbm,tikz-cd,verbatim,cases,thmtools}
\usepackage{mathrsfs}
\usepackage{tikz-cd}
\usepackage[utf8]{inputenc}
\PassOptionsToPackage{hyphens}{url}\usepackage{hyperref}\PassOptionsToPackage{hyphens}{url}\usepackage{hyperref}

\usepackage{fancyhdr}
\allowdisplaybreaks

\newtheorem{thm}{Theorem}[section]
\newtheorem{lem}[thm]{Lemma}
\newtheorem{prop}[thm]{Proposition}
\newtheorem{cor}[thm]{Corollary}
\theoremstyle{definition}

\newtheorem{rem}[thm]{Remark}
\newtheorem{defn}[thm]{Definition}
\newtheorem*{defn*}{Definition}
\newtheorem*{thm*}{Theorem}
\newtheorem*{prop*}{Proposition}
\newtheorem*{cor*}{Corollary}

\newtheorem*{question*}{Question}

\newcommand{\cE}{\mathcal{E}}

\newcommand{\cL}{\mathcal{L}}
\newcommand{\cO}{\mathcal{O}}
\newcommand{\cS}{\mathcal{S}}

\newcommand{\C}{\mathbb{C}}

\newcommand{\R}{\mathbb{R}}

\newcommand{\Z}{\mathbb{Z}}

\newcommand{\pqq}{\mathsf{q}}
\newcommand{\pp}{\mathsf{p}}

\newcommand{\rot}{\mathsf{R}}
\newcommand{\refl}{{\underline{\mathsf{R}}}}

\DeclareMathOperator{\Index}{index}
\DeclareMathOperator{\nullity}{nullity}

\newcommand{\Nul}{\operatorname{Nul}}
\newcommand{\Ind}{\operatorname{Ind}}
\newcommand{\Ric}{\operatorname{Ric}}
\newcommand{\tensor}{\otimes}

\newcommand{\Id}{\operatorname{Id}}

\newcommand{\Span}{\operatorname{Span}}

\newcommand{\im}{\operatorname{Im}}

\newcommand{\id}{\operatorname{id}}

\newcommand{\sff}{{\rm II}}

\newcommand{\oplusu}{{\,{\oplus}_{\cL}\,}} 
\newcommand{\Omu}{{\underline{\Omega}}}

\begin{document}
 
% --------------------------------------------------------------
%                         Start here
% --------------------------------------------------------------

\title{Index of minimal hypersurfaces in real projective spaces}

\address{Department of Mathematics, Stanford University, 450 Jane Stanford Way, Bldg
380, Stanford, CA 94305}

\author{Shuli Chen}
\email{shulic@stanford.edu}

\date{\today}
 \maketitle

\begin{abstract} We prove that for an embedded unstable one-sided minimal hypersurface of the $(n+1)$-dimensional real projective space, the Morse index is at least $n+2$, and this bound is attained by the cubic isoparametric minimal hypersurfaces. We also show that there exist closed embedded two-sided minimal surfaces in the 3-dimensional real projective space of each odd index by computing the index of the Lawson surfaces. 
\end{abstract}
\section{Introduction}
Minimal hypersurfaces in the unit round sphere $S^{n+1}$ have been intensively studied. For instance, Lawson \cite{lawson1970complete} constructed three families of minimal surfaces in $S^3$, now called the Lawson surfaces, and used them to show that every closed surface but the projective plane can be minimally immersed into $S^3$. However, determining the Morse index of a minimal hypersurface is in general a difficult problem which has been fully solved only in a few cases. For example, Solomon \cite{solomon1990harmonic, solomon1990harmonic2} computed the index of cubic isoparametric minimal hypersurfaces via representation theory. Kapouleas and Wiygul \cite{kapouleas2020index} computed the index of the Lawson surfaces $\xi_{m-1,1}$ using their symmetries.

On the other hand, certain bounds can be given on the index of minimal hypersurfaces in $S^{n+1}$. In the late '60s, Simons \cite{simons1968minimal} proved that there do not exist closed stable minimal hypersurfaces in $S^{n+1}$ and also characterized the totally geodesic hyperspheres as the only closed minimal hypersurfaces
of index one. Later, Urbano \cite{urbano1990minimal} proved that any closed two-sided minimal surface in $S^3$ that is not totally geodesic has index at least 5 and the equality occurs only for the minimal Clifford surface. A direct computation shows that minimal Clifford  hypersurfaces have index $n+3$ in $S^{n+1}$, and it can be shown that any closed two-sided minimal non-totally geodesic hypersurface in $S^{n+1}$ has
index bigger than or equal to $n + 3$ (see e.g., \cite{perdomo2001low}). An open question is whether the minimal Clifford hypersurfaces are the only minimal immersed hypersurfaces in $S^{n+1}$ with index $n+3$. For a two-sided minimal non-totally geodesic hypersurface $\Sigma^n$ in $S^{n+1}$, Savo \cite{savo2010index} showed that its index is bounded below by $\frac{2}{(n+1)(n+2)}b_1(\Sigma) + n + 2$, which is a linear function of its first Betti number $b_1(\Sigma)$. In this paper, we observe that Savo's bound can be extended to the one-sided case as well (see Theorem~\ref{thm: one-sided, sphere, Betti}).

Since the real projective space $\R P^{n+1}$ is the quotient of $S^{n+1}$ by the antipodal map, it is also natural to study minimal hypersurfaces in $\R P^{n+1}$. Notice that $\R P^{n+1}$ has no closed stable two-sided minimal hypersurfaces because it has positive Ricci curvature. However, the real projective subspaces $\R P^{n}$ are closed stable one-sided minimal hypersurfaces in $\R P^{n+1}$, and Ohnita \cite{ohnita1986stable} proved that they are the only stable ones; this demonstrates that the stability condition in the one-sided case in general becomes more subtle. 
For two-sided minimal hypersurfaces in $\R P^{n+1}$, do Carmo, Ritor\'e, and Ros \cite{do2000compact} proved that the totally geodesic sphere and the Clifford hypersurfaces are the only closed two-sided minimal hypersurfaces of $\R P^{n+1}$ with index 1, and Batista and Martins \cite{batista2022minimal} further showed that they are the only closed two-sided minimal and constant scalar curvature
hypersurfaces in $\R P^{n+1}$ with index less than or equal to two. Ambrozio, Carlotto, Sharp \cite{ambrozio2018comparing} extended Savo's result on hypersurfaces in $S^{n+1}$ to hypersurfaces in $\R P^{n+1}$, and showed that for a closed embedded minimal hypersurface $\Sigma \subset \R P^{n+1}$, its index is bounded below by $\frac{2}{(n+1)(n+2)}b_1(\Sigma)$.

In this paper, we consider one-sided minimal hypersurfaces of $\R P^{n+1}$ and show that unlike the two-sided case, there is a large gap in the index.
We also compute the index of the cubic isoparametric minimal hypersurfaces in $\R P^{n+1}$ using works of Solomon \cite{solomon1990harmonic, solomon1990harmonic2}. See Section \ref{subsec: isoparametric} for the definition of cubic isoparametric minimal hypersurfaces.

\begin{thm}\label{thm: RP^n}
For an embedded unstable one-sided minimal hypersurface $\Sigma^n$ of $\R P^{n+1}$, the Morse index is at least $n+2$. Further, this bound is attained by the cubic isoparametric minimal hypersurfaces.
\end{thm}

In addition, we establish the following theorem by computing the index of the Lawson surfaces, using work of Kapouleas and Wiygul \cite{kapouleas2020index}:

\begin{thm}\label{thm: RP^3}
There exist closed embedded two-sided minimal surfaces of every odd index in $\R P^3$.
\end{thm}

 \subsection*{Acknowledgments.}
  The author is grateful for the useful discussions and suggestions of Otis Chodosh, Alejandra Ramirez-Luna, Brian White, and Yujie Wu. The author also wants to thank the reviewer for their careful reading of the manuscript and their detailed and constructive comments. The author is partially sponsored by the Ric Weiland Graduate Fellowship at Stanford University.

\section{Preliminaries}
\subsection{Schr\"odinger-type operators on involutive manifolds.}
Let $(\Sigma^n,g)$ be a closed Riemannian manifold. Let $V \in C^\infty(\Sigma)$ and consider the Schr\"odinger operator $-\Delta + V$. Then standard elliptic theory tells us that the eigenvalues of $L = -\Delta + V$ are discrete and bounded from below, and we can list them in increasing order as $\lambda_1 < \lambda_2 \le \lambda_3 \le \dots \to \infty$. It is well-known that we have the following variational characterization of the eigenvalues: 
\begin{align}\label{eqn: max-min}
\lambda_{k}=\max_{S\in S_{k-1}}\min_{u\in S^{\perp}, u \neq 0}\frac{\int_{\Sigma}|\nabla u|^2 + Vu^2\, d\Sigma}{\int_{\Sigma}u^2\, d\Sigma},
\end{align}
where $S_{k-1}$ denotes the collection of all the subspaces of $W^{1,2}(\Sigma)$ of dimension $k-1$ and $S^\perp$ denotes the $L^2$-orthogonal complement of $S$ in $W^{1,2}(\Sigma)$. 
This value can be attained and any nonzero function $u$ attaining this value is an eigenfunction with eigenvalue $\lambda_k$.
\begin{lem}
Let $V_1, V_2 \in C^\infty(\Sigma)$ be such that $V_1 \le V_2$ everywhere and $V_1 \not \equiv V_2$. Consider two Schr\"odinger operators $L_1 = -\Delta + V_1$ and $L_2 = -\Delta + V_2$. Let $\lambda_k(L_i)$ be the $k^\text{th}$ eigenvalue of the operator $L_i$. Then we have $\lambda_k(L_1) < \lambda_k(L_2)$ for all $k$.
\end{lem}
\begin{proof}
Since for any function $u \in W^{1,2}(\Sigma)$, we have
$\int_{\Sigma}|\nabla u|^2 + V_1u^2\, d\Sigma \le \int_{\Sigma}|\nabla u|^2 + V_2u^2\, d\Sigma$, it follows from the above variational characterization of the eigenvalues that $\lambda_k(L_1) \le \lambda_k(L_2)$ for all $k$. 

Now for the sake of contradiction suppose for some $k$ that $\lambda_k(L_1) = \lambda_k(L_2) := \Lambda$. 

Let $\cS : = \{S \in S_{k-1} \mid S \text{ attains the value $\Lambda$ in \eqref{eqn: max-min} for $L_2$}\}$.

Then for any space $\tilde{S} \in S_{k-1}$, $\tilde{S} \not \in \cS$, there exists $0 \neq \tilde{u} \in \tilde{S}$ such that 
$\frac{\int_{\Sigma}|\nabla \tilde{u}|^2 + V_2\tilde{u}^2\, d\Sigma}{\int_{\Sigma}\tilde{u}^2\, d\Sigma} < \Lambda$. Since $V_1 \le V_2$, we have that $\frac{\int_{\Sigma}|\nabla \tilde{u}|^2 + V_1\tilde{u}^2\, d\Sigma}{\int_{\Sigma}\tilde{u}^2\, d\Sigma} < \Lambda$ as well. This shows $\tilde{S}$ cannot attain the value $\Lambda$ in \eqref{eqn: max-min} for $L_1$.

However, since $\Lambda$ is also the $k^\text{th}$ eigenvalue of $L_1$, there exists $S\in \cS$ such that $S$ attains the value $\Lambda$ in \eqref{eqn: max-min} for $L_1$. That is, 
\begin{align}\label{eqn: V_1}\min_{u\in S^{\perp}, u \neq 0} \frac{\int_{\Sigma}|\nabla {u}|^2 + V_1 {u}^2\, d\Sigma}{\int_{\Sigma}{u}^2\, d\Sigma} = \Lambda.\end{align}

On the other hand, from the definition of $\cS$ we have that $S$ also attains the value $\Lambda$ in \eqref{eqn: max-min} for $L_2$. So there exists $0 \neq \hat{u} \in S^\perp$ such that

\begin{align}\label{eqn: V_2}\frac{\int_{\Sigma}|\nabla \hat{u}|^2 + V_2\hat{u}^2\, d\Sigma}{\int_{\Sigma}\hat{u}^2\, d\Sigma} = \min_{u\in S^{\perp}, u \neq 0} \frac{\int_{\Sigma}|\nabla {u}|^2 + V_2 {u}^2\, d\Sigma}{\int_{\Sigma}{u}^2\, d\Sigma} = \Lambda.\end{align}

Combining equations \eqref{eqn: V_1} and \eqref{eqn: V_2} with the inequality $V_1 \le V_2$ shows that 
$$\Lambda = \min_{u\in S^{\perp}, u \neq 0} \frac{\int_{\Sigma}|\nabla {u}|^2 + V_1 {u}^2\, d\Sigma}{\int_{\Sigma}{u}^2\, d\Sigma} \le
\frac{\int_{\Sigma}|\nabla \hat{u}|^2 + V_1\hat{u}^2\, d\Sigma}{\int_{\Sigma}\hat{u}^2\, d\Sigma} \le
\frac{\int_{\Sigma}|\nabla \hat{u}|^2 + V_2\hat{u}^2\, d\Sigma}{\int_{\Sigma}\hat{u}^2\, d\Sigma} = \Lambda,$$
so we must have that
$$\frac{\int_{\Sigma}|\nabla \hat{u}|^2 + V_1\hat{u}^2\, d\Sigma}{\int_{\Sigma}\hat{u}^2\, d\Sigma} =
\frac{\int_{\Sigma}|\nabla \hat{u}|^2 + V_2\hat{u}^2\, d\Sigma}{\int_{\Sigma}\hat{u}^2\, d\Sigma} = \Lambda,$$
i.e., $\hat{u}$ attains the value $\Lambda = \lambda_k(L_1) = \lambda_k(L_2)$ in \eqref{eqn: max-min} for both $L_1$ and $L_2$. 

Thus $\hat u$ is an eigenfunction of both $L_1$ and $L_2$, and $\int_{\Sigma}|\nabla \hat{u}|^2 + V_1\hat{u}^2\, d\Sigma = \int_{\Sigma}|\nabla \hat{u}|^2 + V_2\hat{u}^2\, d\Sigma$. 
Hence we have $\int_{\Sigma}(V_2 - V_1)\hat{u}^2 d\Sigma = 0$. This shows $\hat{u}$ vanishes on the open set $\{V_2 \neq V_1\}$, which by the unique continuation principle (see, e.g., \cite{aronszajn1956unique}) implies $\hat{u} \equiv 0$, a contradiction.

\end{proof}

Now suppose there exists a Riemannian
involution $\tau: (\Sigma, g) \to (\Sigma, g)$, namely a smooth isometry such that
$$\tau \circ \tau = \id, \, \tau \neq \id.$$
Then given a linear function space $X$ consisting of functions defined on $\Sigma$ we introduce the subspaces of even and odd functions with respect to the action of $\tau$:
$$X_\cE = \{\phi \in X \mid \phi \circ \tau = \phi\} , \,  X_\cO = \{\phi \in X \mid \phi \circ \tau = -\phi\}.$$

%As discussed in Section 2.1 of  \cite{ambrozio2019bubbling}, 
In particular, with respect to the involution $\tau$, we have the $L^2$-orthogonal decompositions
\begin{align*}
C^\infty(\Sigma) &= C^\infty(\Sigma)_\cE \oplus C^\infty(\Sigma)_\cO \\
W^{1,2}(\Sigma) &= W^{1,2}(\Sigma)_\cE \oplus W^{1,2}(\Sigma)_\cO,
\end{align*}

Given a function $V\in C^\infty(\Sigma)_\cE$ we study the operator $L = -\Delta + V$. For each eigenvalue $\lambda$, let $W_\lambda$ be the corresponding space of eigenfunctions. Then it is easy to see that we have the $L^2$-orthogonal decomposition $W_\lambda =  (W_\lambda)_\cE \oplus (W_\lambda)_\cO$.

Let $spec(L)$ denote the set of eigenvalues of $L$, and let $spec_\cE(L)$ and $spec_\cO(L)$ denote the set of even eigenvalues and odd eigenvalues, respectively. That is,
$$\lambda \in spec_\cE(L) \Leftrightarrow \exists  u \in C^\infty(\Sigma)_\cE \setminus \{0\} \text{ such that } L u = \lambda u, $$
$$\lambda \in spec_\cO(L) \Leftrightarrow \exists  u \in C^\infty(\Sigma)_\cO \setminus \{0\} \text{ such that } L u = \lambda u. $$
Then $spec(L) = spec_\cE(L) \cup spec_\cO(L)$ (for proof, see, e.g., \cite[Lemma 16]{ambrozio2019bubbling}).
Let $\lambda_k^\cE(L)$ (resp. $\lambda_k^\cO(L)$) be the $k^\text{th}$ even (resp. odd) eigenvalue of the operator $L$. For $\lambda_k^\cE(L)$ (resp. $\lambda_k^\cO(L)$), we also have the variational characterization as in \eqref{eqn: max-min}, with $W^{1,2}(\Sigma)$ replaced by $W^{1,2}(\Sigma)_\cE$ (resp. $W^{1,2}(\Sigma)_\cO$).
Using this, it is straightforward to deduce
\begin{lem}\label{lem: even and odd comparison}
Let $(\Sigma^n,g)$ be an $n$-dimensional closed Riemannian manifold with a Riemannian involution $\tau$. Let $V_1, V_2 \in C^\infty(\Sigma)_\cE$ be such that $V_1 \le V_2$ everywhere and $V_1 \not \equiv V_2$. Consider two Schr\"odinger operators $L_1 = -\Delta + V_1$ and $L_2 = -\Delta + V_2$. Let $\lambda_k^\cE(L_i)$ (resp. $\lambda_k^\cO(L_i)$) be the $k^\text{th}$ even (resp. odd) eigenvalue of the operator $L_i$. Then we have $\lambda_k^\cE(L_1) < \lambda_k^\cE(L_2)$ and $\lambda_k^\cO(L_1) < \lambda_k^\cO(L_2)$ for all $k$.
\end{lem}

\subsection{Morse index of minimal submanifolds}
Let $\Phi:(\Sigma^n,g) \rightarrow (M^{m},\bar{g})$ be a minimal isometric immersion of an $n$-dimensional closed Riemannian manifold $(\Sigma,g)$. Let $C^\infty(\Phi^*TM))$ denote the space of all $C^\infty$-vector fields along $\Phi$. For any $V\in C^\infty(\Phi^*TM)$, let $\{\Phi_t\}$ be a $C^\infty$-one-parameter family of immersions of $\Sigma$ into $M$ such that $\Phi_0 = \Phi$ and $\frac{d}{dt}\Phi_t(x)\vert_{t = 0} =
V_x$ for each $x \in \Sigma$. Then we have the classical second variational formula

\begin{thm}[Second variation formula]
The second variation of the volume $|\Phi_t(\Sigma)|$ at time $t = 0$ is given by
\begin{align*}%\label{SVF}
\frac{d^{2}}{dt^{2}} |\Phi_t(\Sigma)|\biggr\vert_{t = 0}=-\int_{\Sigma}\langle  \cL_{\Sigma}V^N,V^N \rangle d\Sigma,
\end{align*}
where $V^N$ is the normal
component of $V$, $\cL_{\Sigma}$ is the elliptic Jacobi operator on the normal bundle $N(\Sigma)$ defined by
$$\cL_{\Sigma}(X):=\Delta^{\perp}X+(\sum_{i=1}^{n}R^{M}(X,e_{i})e_{i})^{\perp}+\sum_{i,j=1}^{n} \langle  \sff_\Sigma(e_{i},e_{j}),X \rangle \sff_\Sigma(e_{i},e_{j}),$$
and ${\Delta}^{\perp}$ is the normal Laplacian on the normal bundle $N(\Sigma)$ given by
$${\Delta}^{\perp}X=\sum_{i=1}^{n}(\overline{\nabla}^{\perp}_{e_{i}}\overline{\nabla}^{\perp}_{e_{i}}X-\overline{\nabla}^{\perp}_{(\overline{\nabla}_{e_{i}}e_{i})^{T}}X).$$
Here, $\{e_{1},\ldots,e_{n}\}$ is an orthonormal basis of $T\Sigma$, $\overline{\nabla}$ is the connection of $M$, $\overline{\nabla}^{\perp}$ is the normal connection of $\Sigma$ in $M$, $\sff_\Sigma$ is the second fundamental form of $\Sigma$ in $M$, and $R^{M}$ is the curvature tensor of $M$. 
\end{thm}

Now for $V \in C^\infty(N\Sigma)$, we define the quadratic form $Q_\Sigma$ on $C^\infty(N\Sigma)$ by 
$$Q_\Sigma(V, V) := -\int_{\Sigma}\langle  \cL_{\Sigma}V,V \rangle d\Sigma.$$

For $\Sigma$, we define its \emph{Morse index} $\Ind(\Sigma)$ to be the largest dimension of a linear subspace of $C^\infty(N\Sigma)$ where $Q_\Sigma$ is negative definite.

\subsection{Morse index of minimal hypersurfaces}
In the special case where $\Sigma$ is a two-sided closed minimal hypersurface of $M$, we let $N$ be a unit normal vector field of $\Sigma$. Notice that if the ambient space $M$ is orientable, then two-sideness of $\Sigma$ is equivalent to $\Sigma$ being orientable. For any $V$ in $C^\infty(\Phi^*TM)$, let $\{\Phi_t\}$ be a $C^\infty$-one-parameter family of immersions of $\Sigma$ into $M$ such that $\Phi_0 = \Phi$ and $\frac{d}{dt}\Phi_t(x)\vert_{t = 0} =
V_x$ for each $x \in \Sigma$. 
Then there exists a unique $C^\infty$ function $u_V: \Sigma \to \R$ such that $V^N = u_V N$. And the second variation formula simplifies to 
\begin{center}
    $\displaystyle \frac{d^{2}}{dt^{2}} |\Phi_t(\Sigma)|\biggr\vert_{t = 0}=-\int_{\Sigma}\langle  \cL_{scal} u_V, u_V\rangle d\Sigma$,
\end{center}
where $\cL_{scal}: = \Delta + \Ric_M(N,N) + |\sff_\Sigma|^2$ is the Jacobi
operator on $\Sigma$ acting on scalar-valued functions. %and $\sff_\Sigma$ denotes the second fundamental form of $\Sigma$ in $M$.

Now for $u \in C^\infty(\Sigma)$, we can define the quadratic form $Q_{\Sigma,scal}$ on $C^\infty(\Sigma)$ by 
$$Q_{\Sigma,scal}(u, u) := -\int_{\Sigma}\langle \cL_{scal}u,u \rangle d\Sigma.$$

For $\Sigma$, we define its \emph{Morse index}, $\Ind(\Sigma)$, to be the largest dimension of a linear subspace of $C^\infty(\Sigma)$ where $Q_{\Sigma,scal}$ is negative definite. By the theory of elliptic operators on closed manifolds, the Morse index also equals the number of negative eigenvalues (counted with multiplicity) of the Jacobi operator $\cL_{scal}$. 
It is straightforward to check that the definition given here agrees with the definition of Morse index in the general case. In the following, we will omit the subscript $scal$ when there is no confusion. 

We call a $C^2$-function $u\colon \Sigma \to \R$ a \emph{Jacobi field} if it satisfies $\cL u = 0$ on $\Sigma$. 
We define the \emph{nullity} of $\Sigma$, $\Nul(\Sigma)$, as the dimension of the space of Jacobi fields on $\Sigma$. Equivalently, this is the number of $0$-eigenvalues (counted with multiplicity) of $\cL$.  It is well-known that Killing vector fields of the ambient manifold $M$ induce Jacobi fields of $\Sigma$. We call a
Jacobi field on $\Sigma$ \emph{non-exceptional} if it is induced by a Killing field; otherwise we call it \emph{exceptional}.

%Let $\mathcal{K}(\Sigma)$ be the space of Jacobi fields on $\Sigma$. We define the \emph{nullity} of $\Sigma$, $\Nul(\Sigma)$, as the dimension of $\mathcal{K}(\Sigma)$, that is, the number of $0$-eigenvalues (counted with multiplicity) of $\cL$. 

Now suppose on $\Sigma$ there exists a Riemannian
involution $\tau: (\Sigma, g) \to (\Sigma, g)$.
We define the \emph{even Morse index} $\Ind_\cE(\Sigma)$ (resp. the \emph{odd Morse index} $\Ind_\cO(\Sigma)$) as the largest dimension
of a linear subspace of $C^\infty(\Sigma)_\cE$ (resp. $C^\infty(\Sigma)_\cO$) where $Q_{\Sigma}$ is negative definite. Again, the even (resp. odd) Morse index  equals the number of even (resp. odd) negative eigenvalues  of $\cL$, counted with multiplicity.  
We also define the \emph{even nullity} $\Nul_\cE(\Sigma)$ (resp. the \emph{odd nullity} $\Nul_\cO(\Sigma)$) as the dimension of the space of even (resp. odd) Jacobi fields on $\Sigma$.
Then it is not hard to deduce that (see, e.g., \cite[Section 2.1]{ambrozio2019bubbling})
$$\Ind(\Sigma) = \Ind_\cE(\Sigma) + \Ind_\cO(\Sigma),$$
$$\Nul(\Sigma) = \Nul_\cE(\Sigma) + \Nul_\cO(\Sigma).$$

\section{Minimal hypersurfaces in $S^{n+1}$} \label{sec: sphere}
Let $({\Sigma}^{n}, g) \to (S^{n+1}, \overline{g})$ be a closed connected orientable (hence two-sided) immersed minimal hypersurface of the $(n+1)$-dimensional
Euclidean unit sphere. Let $N$ denote a unit normal vector field on $\Sigma$. Let $\overline{\nabla}$ denote the Levi-Civita connection on $S^{n+1}$, let $\nabla$ denote the induced Levi-Civita connection on $\Sigma$, and let $\Delta$ denote the Laplacian on $\Sigma$. 
Then $\Ric_{S^{n+1}}(N,N) = n$ and the Jacobi operator on $\Sigma$ becomes $\cL = \Delta + |\sff_\Sigma|^2 + n$.

\begin{defn}[Test functions $l_w$, $f_w$, $r_{v,w}$]\label{def: test fcns}
For every fixed vector $w \in \R^{n+2}$, we define functions $l_w: \Sigma \to \R$ and $f_w: \Sigma \to \R$ by 
$$l_w(x) = \langle w, x\rangle,$$
$$f_w(x) =  \langle w, N(x)\rangle$$
for all $x \in  \Sigma$. Here $\langle ,\rangle$ denotes the inner product in $\R^{n+2}$.
For $v,w \in \R^{n+2}$, we further define $r_{v,w}: \Sigma \to \R$ by
$$r_{v,w}: = l_vf_w - l_wf_v.$$
\end{defn}

A standard computation using Gauss and Weingarten formulas shows that
\begin{align*}
%\label{eqn: gradient l_w} 
\nabla l_w &= w^\top,\\
%\label{eqn: gradient f_w} 
\nabla f_w &= -A(w^\top),
\end{align*}
where $\top$ denotes the projection onto $T\Sigma$ and $A(X) = - (\overline{\nabla}_X N)^\top$ is the shape operator of $\Sigma$.

Using the minimality of $\Sigma$ and the Codazzi equations, we have 
\begin{align*}
%\label{eqn: l_w} 
\Delta l_w &= -nl_w,\\
%\label{eqn: f_w} 
\Delta f_w &= -|\sff_\Sigma|^2f_w.
\end{align*}
Then each nonzero $f_w$ is an eigenfunction of $\cL$ with eigenvalue $-n$.

We define the linear subspaces $\Gamma_1, \Gamma_2, \Gamma_3$ of $C^\infty(M)$ by
$$\Gamma_1: = \{l_w \mid w \in \R^{n+2}\},$$
$$\Gamma_2: = \{f_w \mid w \in \R^{n+2}\},$$
$$\Gamma_3: = \{r_{v,w}\mid v,w \in \R^{n+2}\}.$$

We have the following three simple lemmas regarding these spaces:

\begin{lem}\label{lem: lw}
For $0 \neq w \in \R^{n+2}$, we have $Q_{\Sigma}(l_w, l_w) \le 0$, and $Q_{\Sigma}(l_w,l_w) = 0$ if and only if $\Sigma$ is an equatorial hypersphere. Consequently, for $\Sigma$ not an equatorial hypersphere, $Q_\Sigma$ is negative definite on $\Gamma_1$ and $\dim \Gamma_1 = n+2$.
\end{lem}
\begin{proof}
Using $\Delta l_w = -nl_w$, we have
$$\cL l_w = \Delta l_w  + |\sff_{\Sigma}|^2 l_w +n l_w = |\sff_\Sigma|^2 l_w,$$
so
$$Q_{\Sigma}(l_w,l_w) = -\int |\sff_{\Sigma}|^2 l_w^2 \le 0.$$

If $Q_{\Sigma}(l_w,l_w) = 0$, then $|\sff_{\Sigma}|^2 l_w^2 \equiv 0$. Thus at every point $x \in \Sigma$, either $|\sff_{\Sigma}| = 0$ or $l_w = \langle w, x \rangle = 0$. If $|\sff_{\Sigma}| \not \equiv 0$, then consider the open neighborhood $U$ where $|\sff_{\Sigma}| > 0$. In $U$ we have $\langle w, x \rangle = 0$, so $U$ is contained in the equatorial hypersphere orthogonal to $w$. Then $U$ is totally geodesic, showing $|\sff_{\Sigma}| = 0$ on $U$, which is a contradiction. Hence the condition $Q_{\Sigma}(l_w,l_w) = 0$ implies $|\sff_{\Sigma}| \equiv 0$, that is, $\Sigma$ is totally geodesic. The only such immersed hypersurface is an equator.
\end{proof}

\begin{lem}\cite[Lemma~3.1]{perdomo2001low}\label{lem: fw}
For $\Sigma$ not an equatorial hypersphere, we have $\dim \Gamma_2 = n+2$.
\end{lem} 

\begin{lem}\label{lem: rvw}
$r_{v,w}$ are non-exceptional Jacobi fields of $\Sigma$.
\end{lem}
\begin{proof}
 We compute
 \begin{align*}
 \Delta(l_vf_w) &=  (\Delta l_v)f_w + l_v\Delta f_w + 2\langle \nabla l_v, \nabla f_w \rangle \\
 & = -nl_vf_w - |\sff_{\Sigma}|^2 l_vf_w + 2 \langle v^\top, - A(w^\top) \rangle \\
 & = -nl_vf_w - |\sff_{\Sigma}|^2 l_vf_w - 2\, \sff_{\Sigma}(v^\top,w^\top).
 \end{align*}
 Thus $\cL(l_vf_w) = - 2\, \sff_{\Sigma}(v^\top,w^\top)$. Using the symmetry of the second fundamental form, we have $\cL(r_{v,w}) = \cL(l_vf_w - l_wf_v) = - 2\, \sff_{\Sigma}(v^\top,w^\top) +  2\, \sff_{\Sigma}(w^\top,v^\top) = 0$ as desired. 
\end{proof}

%\begin{rem}
%The Jacobi fields $r_{v,w}$ come from rotations of the sphere. 
%\end{rem}

Recall that every connected manifold $\Sigma$ admits a double cover $\tilde{\Sigma} \to \Sigma$ by the manifold $\tilde{\Sigma} =  \{(p, \cO_p) \mid  \text{$p \in \Sigma$ and $\cO_p$ is an orientation of $T_p\Sigma$}\}$. We call $\tilde{\Sigma}$ the orientation double cover of $\Sigma$. When $\Sigma$ is non-orientable, $\tilde{\Sigma}$ is connected, and $\tilde{\Sigma} \to \Sigma$ is a two-sheeted smooth covering map. Using the above lemmas, we can give a bound on the index of one-sided (hence non-orientable) minimal hypersurfaces in $S^{n+1}$ and its orientation double cover:

\begin{thm}\label{thm: one-sided, S^{n+1}} 
Let $\Phi:\Sigma \to S^{n+1}$ be a one-sided closed connected immersed minimal hypersurface in $S^{n+1}$ and let $\tilde{\Phi}: \tilde{\Sigma} \to S^{n+1}$ be its orientation double cover. Then $\Ind(\Sigma) = \Ind_\cO(\tilde{\Sigma}) \ge n+2$ and $\Ind(\tilde{\Sigma}) \ge 2n + 5$. 
\end{thm}
\begin{proof}
Since $\Phi:\Sigma \to S^{n+1}$ is a one-sided closed connected immersed minimal hypersurface in $S^{n+1}$, we have that $n \ge 2$ and $\Sigma$ is non-orientable, and we can consider the connected orientation double cover $p: \tilde{\Sigma} \to \Sigma$.  The immersion $\Phi: \Sigma \to S^{n+1}$ lifts to a two-sided immersion $\tilde{\Phi}: \tilde{\Sigma} \to S^{n+1}$
that is also minimal. Since $\tilde{\Phi}$ factors through the covering map $p$, it cannot be an embedding. In particular, this implies $\tilde{\Sigma}$ cannot be the equatorial hypersphere, so $|\sff_{\tilde{\Sigma}}| \not \equiv 0$. Let $N$ denote a unit normal vector field of $\tilde{\Sigma}$ and let $s: \tilde{\Sigma} \to \tilde{\Sigma}$ be the change of sheet involution induced by the covering. Then $N$ is odd with respect to the involution $s$. Normal vector fields on $\Sigma$ are in one-to-one correspondence with normal vector fields on $\tilde{\Sigma}$ that are even, i.e., vector fields of the form $\phi N$ where $\phi$ is an \emph{odd} function with respect to $s$. Hence $\Ind(\Sigma) = \Ind_\cO(\tilde{\Sigma})$. Note that the functions $l_w$ are even with respect to $s$, while the functions $f_w$ are odd with respect to $s$. That is, $l_w \circ s = l_w$ and $f_w \circ s = -f_w$.

Now consider the operator $L_2 = -\Delta - n$. It has even eigenvalue $-n$ with eigenfunction $1$, and even eigenvalue $0$ with eigenfunctions $l_w$. Thus by Lemma~\ref{lem: lw}, $L_2$ has at least $n+3$ even eigenvalues less than or equal to $0$. Since $- n \le - |\sff_{\tilde{\Sigma}}|^2 - n$ and $|\sff_{\tilde{\Sigma}}|^2 \not \equiv 0$, by Lemma~\ref{lem: even and odd comparison}, we have that $\cL_{\tilde\Sigma}$ has at least $n+3$ even eigenvalues strictly less than $0$. In particular, this shows $\Ind_\cE(\tilde\Sigma) \ge n+3$.

By Lemma \ref{lem: fw}, $\Ind_\cO (\tilde{\Sigma}) \ge \dim \Gamma_2 = n+2$. Consequently, $\Ind(\Sigma) = \Ind_\cO(\tilde{\Sigma}) \ge n+2$ and $\Ind(\tilde{\Sigma}) \ge \Ind_\cE (\tilde{\Sigma}) + \Ind_\cO (\tilde{\Sigma}) \ge 2n+5$.
\end{proof}

In \cite{savo2010index}, Savo proved a lower bound on the index of closed orientable minimal hypersurfaces $\Sigma^n$ in $S^{n+1}$ in terms of the first Betti number. Specifically, he proved that for $\Sigma$ not totally geodesic, the number of eigenvalues of $\cL$ less than $-n$ is bounded from below by $$\frac{2}{(n+2)(n+1)}\cdot \text{(dimension of the space of harmonic 1-forms on $\Sigma$)}.$$
By Hodge theory, the dimension of the space of harmonic 1-forms on $\Sigma$ equals $b_1(\Sigma)$. 
Since $-n$ is an eigenvalue of $\cL$ with multiplicity at least $n+2$ (Lemma~\ref{lem: fw}), it follows that $\Ind(\Sigma) \ge \frac{2}{(n+2)(n+1)}b_1(\Sigma) + n +2$.

Now suppose that $\Sigma$ is non-orientable and consider its orientation double cover $\tilde{\Sigma}$ with change of sheet involution $s$. Looking through Savo's proof, we see that every step carries over if we restrict only to functions and forms on $\tilde{\Sigma}$ that can pass down to $\Sigma$; that is, we replace the eigenvalues of $\cL_{\tilde{\Sigma}}$ by odd eigenvalues and we replace the eigen 1-forms of $\Delta_1$ (the rough Laplacian on 1-forms) by the eigen 1-forms of $\Delta_1$ invariant under $s$. Then we obtain that the number of odd eigenvalues of $\cL_{\tilde{\Sigma}}$ less than $-n$ is bounded from below by $$\frac{2}{(n+2)(n+1)}\cdot  \text{(dimension of the space of harmonic 1-forms on $\tilde{\Sigma}$ that pass down to $\Sigma$)}.$$ This quantity is just $\frac{2}{(n+2)(n+1)}b_1(\Sigma)$. Since $-n$ is an odd eigenvalue of $\cL_{\tilde{\Sigma}}$ with multiplicity at least $n+2$ (Lemma~\ref{lem: fw}), it follows that $\Ind_\cO(\tilde{\Sigma}) \ge \frac{2}{(n+2)(n+1)}b_1(\Sigma) + n +2$.
Combining with Proposition~\ref{thm: one-sided, S^{n+1}}, we can extend Savo's result to the non-orientable setting as below.

\begin{thm}\label{thm: one-sided, sphere, Betti} 
Let $\Phi:\Sigma \to S^{n+1}$ be a one-sided closed connected immersed minimal hypersurface in $S^{n+1}$, and let $\tilde{\Phi}: \tilde{\Sigma} \to S^{n+1}$ be its orientation double cover. Let $b_1(\Sigma)$ be the first Betti number of $\Sigma$. Then $\Ind(\Sigma) = \Ind_\cO(\tilde{\Sigma}) \ge \frac{2}{(n+2)(n+1)}b_1(\Sigma) + n+2$ and $\Ind(\tilde{\Sigma}) \ge \frac{2}{(n+2)(n+1)}b_1(\Sigma) + 2n + 5$. 
\end{thm}

\section{Minimal hypersurfaces in $\R P^{n+1}$}\label{sec: RPn}
%View $S^{n+1}$ as the set $\{x \in \R^{n+2} \mid |x| =1\}$ in $\R^{n+2}$. 
Define $\tau = -\Id$ on $\R^{n+2}$. Then $\tau$ restricted to the antipodal map on $S^{n+1}$. Let $({\widehat\Sigma}^{n}, g) \to (S^{n+1}, \overline{g})$ be a closed connected orientable (hence two-sided) immersed minimal hypersurface with antipodal symmetry. 

Let $p$ denote the quotient map $p: S^n \to \R P^{n+1} = S^{n+1}/(x \sim \tau(x))$. Then $\widehat\Sigma$ passes through the quotient map to a closed connected minimal hypersurface $\Sigma$ in $\R P^{n+1}$, and $\widehat \Sigma$ is a double cover of $\Sigma$, with $\tau :\widehat \Sigma \rightarrow \widehat\Sigma$ the change of sheet involution of $\widehat \Sigma$. Let $N$ be a unit normal vector field on $\widehat \Sigma$.

If $(\Sigma^{n}, g)$ is two-sided, then $\tau:\widehat \Sigma \rightarrow \widehat\Sigma$ must satisfy $d\tau_{x}(N(x)) = N(\tau(x))$ for all $x \in \widehat\Sigma$. Since $d\tau_x = -\Id$ on $T_x(\R^{n+2})$, if we consider the normal vector field $N$ of $\widehat \Sigma$ as a map to $\R^{n+2}$ then $N(\tau(x)) = d\tau_{x}(N(x)) = -N(x)$, showing that $N$ is odd under $\tau$. Since $\Sigma$ is two-sided, normal vector fields on $\Sigma$ are in one-to-one correspondence with normal vector fields on $\tilde{\Sigma}$ that are odd with respect to $\tau$, i.e., vector fields of the form $\phi N$ where $\phi$ is an \emph{even} function with respect to $\tau$. 
The Morse index of $\Sigma$ thus equals the index of the quadratic form $Q_{\widehat\Sigma}$ over the space of functions $\varphi \in W^{1,2}(\widehat\Sigma)$ satisfying $\varphi \circ \tau = \varphi$. Thus $\Ind(\Sigma) = \Ind_\cE(\widehat\Sigma)$.

If instead $(\Sigma^{n}, g)$ is one-sided, then $\tau:\widehat \Sigma \rightarrow \widehat\Sigma$ must satisfy $d\tau_{x}(N(x)) = -N(\tau(x))$ for all $x \in \widehat\Sigma$. In this case the normal vector field $N$ of $\widehat \Sigma$, considered as a map to $\R^{n+2}$, is even under $\tau$; i.e., $N(\tau(x)) = N(x)$. 
The Morse index of $\Sigma$ equals the index of the quadratic form $Q_{\widehat\Sigma}$ over the space of functions $\varphi \in W^{1,2}(\widehat\Sigma)$ satisfying $\varphi \circ \tau = -\varphi$. Thus $\Ind(\Sigma) = \Ind_\cO(\widehat\Sigma)$.

%Let $(\Sigma^{n}, g) \to (\R P^{n+1}, \overline{g})$ be a connected compact minimal hypersurface of finite index. 

\subsection{Lower bound on the index of one-sided minimal hypersurfaces in $\R P^{n+1}$}
When $n$ is even, $\R P^{n+1}$ is orientable. Then a connected one-sided hypersurface in $\R P^{n+1}$ is non-orientable, while a connected two-sided hypersurface in $\R P^{n+1}$ is orientable. 
When $n$ is odd, $\R P^{n+1}$ is non-orientable. Then a connected one-sided hypersurface in $\R P^{n+1}$ is orientable, while a connected two-sided hypersurface in $\R P^{n+1}$ is non-orientable, unless it contains
no loop noncontractible in $\R P^{n+1}$. For details, see \cite[Section 1]{viro1998mutual}.

Since $\R P^{n+1}$ has positive Ricci curvature, it does not have closed stable connected two-sided minimal hypersurfaces. On the other hand, Ohnita \cite[Theorem~C]{ohnita1986stable} showed the only closed stable connected one-sided minimal hypersurfaces of $\R P^{n+1}$ are the projective subspaces $\R P^{n}$. Here we show that in many cases there is a large gap on the index of one-sided minimal hypersurfaces of $\R P^{n+1}$:

\begin{thm}\label{thm: RPn, general}
Let $\Phi: \Sigma^{n} \to \R P^{n+1}$ be an unstable connected one-sided immersed minimal hypersurface of $\R P^{n+1}$. Suppose either
\begin{itemize}
\item $n$ is odd, or
\item $n$ is even and $\Phi$ lifts to an immersion of $\Sigma^{n}$ into $S^{n+1}$, or
\item $n$ is even and there exist a connected
orientable twofold covering $\widehat{\Sigma} \to \Sigma$ and an isometric minimal immersion $\widehat{\Phi}: \widehat{\Sigma} \to S^{n+1}$ locally
congruent to $\Phi$ such that $\widehat{\Sigma}$ is invariant under the antipodal symmetry $\tau: S^{n+1} \to S^{n+1}$,
\end{itemize}
Then $\Ind(\Sigma) \ge n+2$. 
\end{thm}
\begin{proof}
\emph{Case 1:} Suppose $n$ is odd. Then since $\Sigma$ is one-sided, $\Sigma$ is orientable. There are two subcases: either (1) $\Phi$ lifts to an immersion of $\Sigma^{n}$ into $S^{n+1}$, or (2) there exist a connected twofold covering $\widehat{\Sigma} \to \Sigma$ and an isometric minimal immersion $\widehat{\Phi}: \widehat{\Sigma} \to S^n$ locally congruent to $\Phi$ such that $\widehat{\Sigma}$ is invariant under the antipodal symmetry $\tau: S^{n+1} \to S^{n+1}$. 

\emph{Subcase (1):} If $\Phi$ lifts to an immersion of $\Sigma$ into $S^{n+1}$, then $\Sigma$ is a one-sided minimal hypersurface in the sphere. However, one-sided minimal hypersurfaces in the sphere are non-orientable, contradicting the fact that $\Sigma$ is orientable. So this subcase cannot happen.

\emph{Subcase (2):} Since Subcase (1) cannot happen, there exist a connected twofold covering $\widehat{\Sigma} \to \Sigma$ and an isometric minimal immersion $\widehat{\Phi}: \widehat{\Sigma} \to S^{n+1}$ locally
congruent to $\Phi$ such that $\widehat{\Sigma}$ is invariant under the antipodal symmetry $\tau: S^{n+1} \to S^{n+1}$, $x \mapsto -x$. 
Since $\Sigma$ is orientable, $\widehat \Sigma$ is also orientable (hence two-sided in $S^{n+1}$). 
Then $\Ind(\Sigma) = \Ind_\cO(\widehat\Sigma)$.

If $\widehat\Sigma$ is an equatorial sphere, then by \cite{simons1968minimal}, $\Ind(\widehat\Sigma) = 1 = \Ind_\cE(\widehat\Sigma)$, and $\Ind_\cO(\widehat\Sigma) = 0$. This shows $\Ind(\Sigma) = \Ind_\cO(\widehat\Sigma) = 0$. This is impossible since we assume $\Sigma$ to be unstable.

Thus $\widehat\Sigma$ is a non-equatorial minimal hypersurface in $S^{n+1}$. Notice that the functions $l_w \in \Gamma_1$ as defined in Definition \ref{def: test fcns} are odd with respect to the antipodal map $\tau$. By Lemma \ref{lem: lw}, $Q_{\widehat\Sigma}$ is negative definite on $\Gamma_1$, so $\Ind_\cO(\widehat\Sigma) \ge \dim(\Gamma_1) = n+2$.

Thus $\Ind(\Sigma) = \Ind_\cO(\widehat\Sigma) \ge n+2$ in this subcase.

\emph{Case 2:} Suppose $n$ is even and $\Phi$ lifts to an immersion of $\Sigma^{n}$ into $S^{n+1}$. Then the normal bundles of $\Sigma^{n}$ in the two ambient spaces $\R P^{n+1}$ and $S^{n+1}$ are isomorphic, so the index of $\Sigma$ in $\R P^{n+1}$ equals the index of $\Sigma$ in $S^{n+1}$. By Theorem~\ref{thm: one-sided, S^{n+1}}, $\Ind(\Sigma) \ge n+2$.

\emph{Case 3:} Suppose $n$ is even and there exist a connected orientable twofold covering $\widehat{\Sigma} \to \Sigma$ and an isometric minimal immersion $\widehat{\Phi}: \widehat{\Sigma} \to S^{n+1}$ locally
congruent to $\Phi$ such that $\widehat{\Sigma}$ is invariant under the antipodal symmetry $\tau$. In this case the reasoning is exactly the same as in Case 1 Subcase (2), and we have $\Ind(\Sigma) \ge n+2$.

\end{proof}

The only possibility not included in the previous theorem is when $n$ is even and the twofold covering $\widehat{\Sigma} \to \Sigma$ invariant under the antipodal symmetry is non-orientable. Note that since embedded hypersurfaces in $S^{n+1}$ (or any simply-connected manifold) are two-sided (see, e.g., \cite{samelson1969orientability}), if we assume $\Sigma$ in $\R P^{n+1}$ to be embedded, then this case actually cannot happen. Therefore we have

\begin{cor}\label{thm: RPn, embeded}
Let $\Phi: \Sigma^{n} \to \R P^{n+1}$ be an unstable connected one-sided embedded minimal hypersurface in $\R P^{n+1}$. Then $\Ind(\Sigma) \ge n+2$.
\end{cor}

\begin{rem}
In \cite{ambrozio2018comparing}, Ambrozio, Carlotto, and Sharp proved for a closed embedded minimal hypersurface $\Sigma \subset \R P^{n+1}$,  $\Ind(\Sigma) \ge \frac{2}{(n+1)(n+2)}b_1(\Sigma)$. They remarked that it is not possible to obtain an extra $n+2$ in the lower bound (see \cite[Remark 5.2]{ambrozio2018comparing}), unlike the case when the ambient space is a sphere (see \cite{savo2010index} and also Theorem \ref{thm: one-sided, sphere, Betti}). 
Indeed, as mentioned earlier, the projective subspace $\R P^{n}$ is one-sided and has index 0 in $\R P^{n+1}$ \cite{ohnita1986stable}, and
the Clifford hypersurfaces are two-sided and have index $1$ in $\R P^{n+1}$ \cite{do2000compact}.
Even when the hypersurface is unstable and one-sided, we have been unable to incorporate the lower bounds  $\frac{2}{(n+1)(n+2)}b_1(\Sigma)$ and $n+2$ together as in the case of Theorem~\ref{thm: one-sided, sphere, Betti}. This is because the functions $l_w$ only produce test functions but not eigenfunctions for the Jacobi operator. 
\end{rem}

\subsection{Index of cubic isoparametric minimal hypersurfaces}\label{subsec: isoparametric}
Next we consider the isoparametric minimal hypersurfaces in $S^{n+1}$. A hypersurface $\Sigma^{n}$ in $S^{n+1}$ is called \emph{isoparametric} if its principal curvature functions, 
$$\kappa_1(x) \le \kappa_2(x) \le \dots \le \kappa_{n}(x),$$
are constants. In particular the norm of the second fundamental form of $\Sigma^{n}$ is constant. Let $g$ be the number of distinct principal curvatures of $\Sigma^{n}$. Isoparametric hypersurfaces were first studied by \'Elie Cartan \cite{cartan1938familles,cartan1939familles,cartan1939quelques} in the late 1930's. The interest in isoparametric hypersurfaces was revived in the 1970s, and in 1973, M{\"u}nzner \cite{munzner1980isoparametrische} established a breakthrough result. He showed that the number of distinct principal curvature $g$ could only be 1, 2, 3, 4, or 6, and there exists a homogeneous polynomial $F$, later called the Cartan--M{\"u}nzner
polynomial, of degree $g$ over $\R^{n+2}$ satisfying 
\begin{gather}
|DF(x)|^2 = g^2|x|^{2g-2},  \\ 
\Delta_{\R^{n+2}} F = c|x|^{g-2} \, \text{ for some constant $c$} \label{eqn: Laplacian of F}  
\end{gather}
%$$|DF(x)|^2 = g^2|x|^{2g-2},$$
%$$\Delta_{\R^{n+2}} F = c|x|^{g-2} \, \text{ for some constant $c$}$$
such that $\Sigma = F^{-1}(a) \cap S^{n+1}$ for some $a \in (-1,1)$. Conversely, each Cartan--M{\"u}nzner polynomial $F$ gives a family of isoparametric hypersurfaces; for any $a \in (-1,1)$, the level set $F^{-1}(a) \cap S^{n+1}$ is smooth and isoparametric. The mean curvature of the hypersurface $F^{-1}(a) \cap S^{n+1}$ varies monotonically between $-\infty$ and $+\infty$ as $a$ varies between $-1$ and $1$. In particular, there is a unique $t \in (-1,1)$ such that the hypersurface $\Sigma = F^{-1}(t) \cap S^{n+1}$ has zero mean curvature, i.e., it is minimal. Such hypersurfaces are called \emph{isoparametric minimal hypersurfaces}. For odd $g$, the isoparametric minimal hypersurface in the family is given by $t=0$. We refer the reader to the book \cite{cecil2015geometry} or the surveys \cite{thorbergsson2000survey, chi2020isoparametric} for more details and references on isoparametric hypersurfaces. 
%The levels sets $F^{-1}(\pm 1) \cap S^{n+1}$ are also interesting; they are smooth minimal submanifolds of higher codimension in $S^{n+1}$, known as the focal varieties of the family.

Let $\Sigma^{n}=F^{-1}(t) \cap S^{n+1}$ be an isoparametric minimal hypersurface with $g$ distinct principal curvatures. Since the Cartan--M{\"u}nzner
polynomial $F$ is homogeneous and $t=0$ for odd $g$, we see that $\Sigma$ is invariant under the antipodal map $\tau$. On $\Sigma$, one can take a unit normal vector field as $N = \frac{1}{g}DF$, where $DF$ is the gradient of $F$ as a function on $\R^{n+2}$. When $g = 2,4,6$, we see that $N$ is odd under $\tau$ as a function to $\R^{n+2}$, while when $g=1,3$,
$N$ is even under $\tau$. Thus when $g=2,4,6$, $\Sigma$ quotients through $\tau$ to a two-sided hypersurface $\Sigma$ of $\R P^{n+1}$ and when $g=1,3$, $\Sigma$ quotients through $\tau$ to a one-sided hypersurface $\Sigma$ of $\R P^{n+1}$. 

\begin{lem}{\cite[Corollary 1]{peng1983minimal}}\label{lem: sff}
Let $\Sigma^{n}$ be an isoparametric minimal hypersurface with $g$ distinct principal curvatures. Then the squared norm of the second fundamental form $|\sff|^2$ equals $(g-1)n$.
\end{lem}
 
Below we focus on cubic isoparametric minimal hypersurfaces $\Sigma$, so $g = 3$. In this case, Cartan \cite{cartan1939familles,cartan1939quelques} showed that the three principal curvatures must have equal multiplicity $m = 1, 2, 4,$ or $8$. Up to rotations, there are exactly four such hypersurfaces; one each of dimension $n = 3, 6, 12$, and $24$. These four hypersurfaces are homogeneous, and the full isometry group is $G= SO(3)$, $SU(3)$, $Sp(3)$, and $F_4$, respectively. The action of $G$ on $\Sigma$ is the restriction of a \emph{linear} action $G \to GL(\R^{n+2})$. Further, $\Sigma = F^{-1}(0) \cap S^{n+1}$ is the zero level set of $F|_{S^{n+1}}$ and $c=0$ in equation \eqref{eqn: Laplacian of F} for the Cartan-M{\"u}nzner polynomial $F: \R^{n+2} \to \R$.

%Now let $\Sigma^{n}$ be a cubic isoparametric minimal hypersurface. Since $F$ is homogeneous, we see that $\Sigma$ is invariant under the antipodal map $\tau$. On $\Sigma$, one can take a unit normal vector field as $N = \frac{1}{3}DF$. This shows $N$ is even under $\tau$. Thus $\Sigma$ quotients through $\tau$ to a one-sided hypersurface $\Sigma$ of $\R P^{n+1}$.
\begin{lem}\label{lem: commute}
Let $\Sigma^{n} \subset S^{n+1}$ be a cubic isoparametric minimal hypersurface and let $G$ be its full isometry group. Let $\tau$ be the antipodal map. Let $\gamma \in G$ and consider the function $\gamma: G \to G$. Then $\gamma$ commutes with $\tau$. 

Thus if $f \in C^\infty(M)$ is even (respectively, odd) with respect to $\tau$, then $\gamma^\#f$ is also even (respectively, odd) with respect to $\tau$, where $\gamma^\#f(x) = f(\gamma^{-1}(x))$. 
\end{lem}
\begin{proof}
Since the action of $G$ on $\Sigma$ is the restriction of a linear action $G \to GL(\R^{n+2})$, we can view $\gamma$ as an element of $GL(\R^{n+2})$. Then it is clear that $\gamma$ commutes with $\tau = -\Id_{n+2}$ as elements of $GL(\R^{n+2})$.
\end{proof}
Let $\Sigma^{n}\subset S^{n+1}$ be a cubic isoparametric minimal hypersurface. By Lemma~\ref{lem: sff}, on $\Sigma$, the squared norm of the second fundamental form $|\sff|^2$ is constant and equals $2n$. The Jacobi operator on $\Sigma^{n}$ is thus $\cL = \Delta + 3n$, and study of the Jacobi operator becomes study of the Laplacian. In \cite{solomon1990harmonic, solomon1990harmonic2}, Bruce Solomon explicitly computed the spectrum of the Laplacian of the four cubic isoparametric minimal hypersurfaces. The functions $1$, $f_w$, $l_w$, $r_{v,w}$ (as defined in Definition~\ref{def: test fcns}) are eigenfunctions corresponding to certain low eigenvalues of the Laplacian, and Solomon showed that all the eigenfunctions can be built up from these functions. Using his result, we can compute the index of the quotients of the four cubic isoparametric minimal hypersurfaces in $\R P^{n+1}$.

\begin{thm}
Let $n = 3,6,12,24$ and let ${\widehat\Sigma}^{n} \subset S^{n+1}$ be a cubic isoparametric minimal hypersurface. Then the quotient $\Sigma = {\widehat\Sigma}^{n}/\{x \sim - x\}$ is a one-sided minimal hypersurface in $\R P^n$ of index $n+2$. 
\end{thm}
\begin{proof}

By earlier discussions in this section, $\Sigma$ is one-sided in $\R P^n$ and $\Ind(\Sigma) = \Ind_\cO(\widehat\Sigma)$. So we just need to show $\Ind_\cO(\widehat\Sigma) = n+2$.
Solomon showed that the only eigenvalues of $\Delta_{\widehat\Sigma}$ less than $3n$ (i.e., the index eigenvalues of $\cL_{\widehat\Sigma}$) are precisely
$0$, $n$, $2n$, $2n+2$ \cite[(2.20)]{solomon1990harmonic}. Of these eigenvalues, $0$ has multiplicity 1 with corresponding eigenfunctions being the constant functions, $n$ has multiplicity $n+2$ with eigenspace $\Gamma_1$, and $2n$ has multiplicity $n+2$ with eigenspace $\Gamma_2$. The functions $l_w \in \Gamma_1$ are odd with respect to $\tau$. Since $N$ is even under the antipodal map $\tau$, the functions $f_w = \langle w, N \rangle \in \Gamma_2$ are also even with respect to $\tau$. Thus $\Ind_\cO(\widehat\Sigma) \ge \dim \Gamma_1 = n+2$, and it only remains to show all the eigenfunctions with eigenvalue $2n+2$ are even.

Solomon \cite[(2.20)]{solomon1990harmonic} showed that there exist complex-valued functions $\zeta \in \Gamma_1 \tensor_\R \C$ and $\nu \in \Gamma_2 \tensor_\R \C$ such that 
$$\Delta \zeta^2 = -(2n+2)\zeta^2 - 2\nu^2,$$
$$\Delta \nu^2 = -(8+4n)\nu^2,$$
so $\Delta(\zeta^2 - \frac{1}{n+3}\nu^2) = -(2n+2)(\zeta^2 - \frac{1}{n+3}\nu^2)$, showing $\zeta^2 - \frac{1}{n+3}\nu^2$ is a complex-valued eigenfunction of the Laplacian with eigenvalue $2n+2$. See Section 2 of \cite{solomon1990harmonic} for definition of $\zeta$ and $\nu$. Further, let $G$ denote the full isometry group of ${\widehat\Sigma}^{n}$. 
The complex eigenspace corresponding to eigenvalue $2n+2$ is precisely
$$\Span_\C\{\gamma^\#(\zeta^2 - \frac{1}{n+3}\nu^2) \mid \gamma \in G \}.$$
It is easy to see that the function $\zeta^2 - \frac{1}{n+3}\nu^2$ is even. Then by Lemma~\ref{lem: commute}, all functions in $\Span_\C\{\gamma^\#(\zeta^2 - \frac{1}{n+3}\nu^2) \mid \gamma \in G \}$ are even. 

Therefore the only odd eigenvalues of $\Delta_{\widehat\Sigma}$ less than $3n$ are $n$ with multiplicity $n+2$, so $\Ind(\Sigma) = \Ind_\cO(\widehat\Sigma) = n+2$ as desired. 

\end{proof}

From the above theorem we deduce
\begin{cor}\label{cor: isoparametric}
The bound $\Ind(\Sigma) \ge n+2$ in Theorem~\ref{thm: RPn, general} and therefore in Corollary~\ref{thm: RPn, embeded} is achieved by the cubic isoparametric minimal hypersurfaces.
\end{cor}
Corollary~\ref{thm: RPn, embeded} and Corollary~\ref{cor: isoparametric} together form Theorem~\ref{thm: RP^n}. 

\begin{comment}
\begin{rem}
For two-sided minimal hypersurfaces, \cite{do2000compact} proved that the totally geodesic sphere and the Clifford hypersurfaces are the only two-sided compact minimal hypersurfaces of $\R P^{n+1}$ with index 1, and \cite{batista2022minimal} further showed that they are the only compact two-sided minimal and constant scalar curvature
hypersurfaces in $\R P^{n+1}$ with Morse index less than or equal to two.
\end{rem}
\end{comment}

\section{Index of minimal surfaces in $\R P^3$}
\subsection{Basic spherical geometry and tessellations}
%In this section we focus on $\R P^3$. 
We first review some basic spherical geometry of $S^3$, following \cite{kapouleas2020index} and \cite{kapouleas2022lawson}.

Given a vector subspace $V$ of the Euclidean space $\R^4$, we denote by $V^\perp$ its orthogonal complement in $\R^4$, and we define the reflection in $\R^4$
with respect to $V$, $\refl_V : \R^4 \to \R^4$, by
$$\refl_V := \Pi_V - \Pi_{V^\perp},$$
where $\Pi_V$ and $\Pi_{V^\perp}$ are the orthogonal projections of $\R^4$ onto $V$ and $V^\perp$
respectively. That is, $\refl_V$ is the linear map which restricts to the identity on $V$ and minus the identity on $V^\perp$. Clearly the fixed point set of $\refl_V$ is $V$.

\begin{defn}[Reflections $\refl _A$]
\label{l:D:refl} 
Given any $A \subset S^3 \subset \R^4$, we define $\refl_A : S^3 \to S^3$ to be the restriction to $S^3$ of $\refl_{\Span(A)}$.
\end{defn}

\begin{defn}[Rotations $\rot_C^\phi$ and Killing fields $K_{C}$]
\label{l:D:rot} 
Given a great circle $C\subset S^3$, angle $\phi \in \R$,
and an orientation chosen on the totally orthogonal circle $C^\perp$,
we define the following: 
\begin{enumerate}[(i)]
\item the rotation about $C$ by angle $\phi$  is the element $\rot_C^\phi$ of $SO(4)$ preserving $C$ pointwise 
and rotating the totally orthogonal circle $C^\perp$ along itself by angle $\phi$ (in accordance with its chosen orientation); 
\item the Killing field $K_{C}$ on $S^3$ is given by 
$\, K_{C}\bigg\rvert_p :=  \frac{\partial}{\partial\phi}\bigg\rvert_{\phi=0} \rot_{C}^\phi(p)$
$\forall p\in S^3$. 
\end{enumerate}
\end{defn}

Let $C_0 = \{(x_1,x_2,0,0) \mid x_1^2 + x_2^2 = 1\} \subset S^3$. Then $C_0^\perp = \{(0,0,x_3,x_4) \mid x_3^2 + x_4^2 = 1\}$.
Let $\pp_0 = (1,0,0,0)$ and $\pp^0 = (0,0,1,0)$. For any angle $\phi \in \R$, we define
$$\pp_\phi := \rot^\phi_{C_0^\perp} \pp_0 = (\cos(\phi), \sin(\phi), 0,0)\hspace{1em} \text{and} \hspace{1em} \pp^\phi := \rot^\phi_{C_0^\perp} \pp^0 = (0,0,\cos(\phi), \sin(\phi)).  $$
We further define for any $\phi \in \R$ the great spheres 
$$\Sigma_\phi := \Span\{C_0^\perp \cup \{\pp_\phi\}\} \cap S^3 \hspace{1em}  \text{and} \hspace{1em}  \Sigma^\phi := \Span\{C_0 \cup \{\pp^\phi\}\} \cap S^3$$
and for any $\phi, \phi' \in \R$ the great circles
$$C_\phi^{\phi'} := \Span\{\pp_\phi, \pp^{\phi'}\} \cap S^3.$$

So in particular, $\Sigma_0 = \{(x_1, 0, x_3, x_4) \mid x_1^2 + x_3^2 + x_4^2 = 1\}$ is the great sphere orthogonal to $\pp_{\pi/2} = (0,1,0,0)$, $\Sigma_{\pi/2}$ is the great sphere orthogonal to $\pp_{0} = (1,0,0,0)$, $\Sigma^0$ is the great sphere orthogonal to $\pp^{\pi/2} = (0,0,0,1)$, and $\Sigma^{\pi/2}$ is the great sphere orthogonal to $\pp^{0} = (0,0,1,0)$.

Note that the four great spheres 
$\Sigma^0$, $\Sigma^{\pi/2}$, $\Sigma_0$, and $\Sigma_{\pi/2}$ 
form a system of four mutually orthogonal two-spheres in $S^3$. 
We will later study the subdivisions made by these two-spheres on $S^3$ and the Lawson surfaces.    
To this end we define 
$\Omu_{**}^{\pm*}$, 
$\Omu_{**}^{*\pm}$, 
$\Omu^{**}_{\pm*}$, 
and 
$\Omu^{**}_{*\pm}$, 
to be the closures of the connected components into which $S^3$ is subdivided by the removal of 
$\Sigma^0$, $\Sigma^{\pi/2}$, $\Sigma_0$, or $\Sigma_{\pi/2}$ respectively, 
chosen so that 
\begin{equation*} 
\label{Omu+} 
\pp^{\pm\pi/2}\in \Omu_{**}^{\pm*}, \quad  
\pp^{\frac{\pi}2\mp\frac{\pi}2}\in \Omu_{**}^{*\pm}, \quad  
\pp_{\pm\pi/2}\in \Omu^{**}_{\pm*}, \quad  
\pp_{\frac{\pi}2\mp\frac{\pi}2}\in \Omu^{**}_{*\pm}.  
\end{equation*} 
To further subdivide we replace $*$'s by $\pm$ signs to denote the corresponding intersections 
of these domains; for example we have  
\begin{equation*} 
\label{Omu+-} 
\Omu_{+-}^{-*} := \Omu_{+*}^{**} \cap \Omu_{*-}^{**} \cap \Omu_{**}^{-*}. 
\end{equation*}  

Clearly we have 
\begin{equation*} 
\label{partialOmu+} 
\partial \Omu_{**}^{\pm*} = \Sigma^0 , \quad  
\partial \Omu_{**}^{*\pm} = \Sigma^{\pi/2} , \quad  
\partial \Omu^{**}_{\pm*} = \Sigma_0 , \quad  
\partial \Omu^{**}_{*\pm} = \Sigma_{\pi/2}, 
\end{equation*} 
and $\Omu^{++}_{++}$ is the spherical tetrahedron 
$\overline{\pp^0 \pp^{\pi/2} \pp_0 \pp_{\pi/2} }$. As a set, $\Omu^{++}_{++} = \{(x_1, x_2, x_3, x_4) \in S^3 \mid x_1, x_2, x_3, x_4 \ge 0\}$.

Let $k \geq 2$, $m \geq 2$ be integers. For $i, j \in \frac{1}{2}\Z$, we define 
$$t_i: = \frac{(2i -1)\pi}{2m}, \hspace{1em} t_j: = \frac{(2j -1)\pi}{2k},$$
$$\pqq_i : = \pp_{t_i} \in C_0, \hspace{1em} \pqq^j : = \pp^{t^j} \in C_0^\perp.$$

Then the $4m$ points $\pqq_{i}$ for $i\in\frac{1}{2}\Z$ subdivide $C_0$ into $4m$ equal arcs of length $\pi/(2m)$ each, 
and the $4k$ points points $\pqq^j$ for $j\in\frac{1}{2}\Z$ subdivide $C_0^\perp$ into $4k$ arcs of length $\pi/(2k)$ each.

Then $\forall i,j \in \frac{1}{2}\Z$ we define
the spherical tetrahedron $\Omega_i^j: = \overline{\, \pqq_i \pqq_{i+1} \pqq^{j} \pqq^{j+1} \, }$ and 
the spherical quadrilateral $Q_i^j\subset \partial \Omega_i^j$  
consisting of the following four edges 
\begin{equation*} 
\label{D:Q} 
Q_i^j :=
\overline{ \pqq_{ i } \pqq^{ j  } }  
\cup 
\overline{ \pqq^{ j } \pqq_{ i+1  } }  
\cup 
\overline{ \pqq_{ i+1 } \pqq^{ j+1  } }  
\cup 
\overline{ \pqq^{ j+1 } \pqq_{ i  } }  
. 
\end{equation*} 
We define  $\underline{\mathbf{\Omega}} := \{\Omega^j_i\}_{i,j\in\frac{1}{2} + \Z}$. Note that $\underline{\mathbf{\Omega}}$ provides tessellations of $S^3$ with $4km$ tetrahedra. The tetrahedra in $\underline{\mathbf{\Omega}}$ are separated by the great spheres $\Sigma_{i\pi/m}$ and $\Sigma^{j\pi/k}$ for $i,j \in \Z$.

For $k=2$ and $m$ even, we have that $\Omu^{++}_{++}$ can be decomposed into $m$ isometric pieces with disjoint interiors as 
\begin{align}
\label{Omu++++} 
\Omu^{++}_{++} =  \bigcup_{i=1}^{\frac{m}{2}} \Omega_{i+\frac{1}{2}}^{\frac{1}{2}}.    
\end{align}

\subsection{The Lawson surfaces $\xi_{m-1,k-1}$}
We briefly discuss the Lawson surfaces $\xi_{m-1,k-1}$ defined in \cite{lawson1970complete}.
\begin{thm}[\cite{lawson1970complete}] 
\label{T:lawson}
Given integers $k \geq 2$, $m \geq 2$, 
and $\forall i,j\in\Z$,  
there is a unique compact connected minimal surface $D_i^j \subset \Omega_i^j$ with 
$\partial D_i^j = Q_i^j$.  
Moreover $D_i^j$ is a disc, minimizing area among such discs, 
and  
$$
\xi_{m-1,k-1} :=\bigcup_{i+j\in2\Z} D_i^j
$$ 
is an embedded connected closed (hence two-sided) smooth minimal surface 
of genus $(k-1)(m-1)$. 
\end{thm} 
For any $k, m$, the Lawson surface $\xi_{m-1,k-1}$ is congruent
to $\xi_{k-1,m-1}$. Taking $k = 1$ with any $m$ gives the great two-sphere
and taking $m = k = 2$ gives the Clifford torus.

If we let $z = x_1 + ix_2$ and $w=x_3 + ix_4$ where $x_1,x_2,x_3, x_4$ are the coordinates for $\R^4$, then by \cite{lawson1970complete}, the group of symmetries for the Lawson surface $\xi_{m-1,k-1}$ corresponds to the group of symmetries in $O(4)$ of the equation 
$$\im(z^m + |w|^{m-k}w^k) = 0.$$

From the symmetry we deduce that when $m,k$ are both even and when $m,k$ are both odd,  $\xi_{m-1,k-1}$ is invariant under the antipodal map. In either case, quotienting by the antipodal map yields an embedded
minimal hypersurface $\overline{\xi}_{m-1,k-1} = \xi_{m-1,k-1}/\{\pm\}$ in the real projective space $\R P^{3} = S^3/\{\pm\}$,
which we will also call a Lawson surface. Since $\xi_{m-1,k-1}$ is a double cover of $\overline{\xi}_{m-1,k-1}$, we have $\chi(\xi_{m-1,k-1}) = 2\chi(\overline{\xi}_{m-1,k-1})$, so $\chi(\overline{\xi}_{m-1,k-1}) = 1-(k-1)(m-1)$.

When $m,k$ are both even, $\overline{\xi}_{m-1,k-1}$ is a two-sided minimal surface of genus $\frac{(k-1)(m-1) + 1}{2}$.

When $m,k$ are both odd, $\chi(\overline{\xi}_{m-1,k-1})$ is odd, so $\overline{\xi}_{m-1,k-1}$ is a non-orientable (hence one-sided) minimal surface of $\R P^3$.

In the following, we will restrict to the case $k=2$ and $m$ even. Then quotienting by the antipodal map yields a two-sided embedded minimal hypersurface $\overline{\xi}_{m-1,1} = \xi_{m-1,1}/\{\pm\}$ of genus $\frac{m}{2}$ in $\R P^{3}$. The reflections $\refl_{\Sigma^{j\pi/2}}, \refl_{\Sigma_{i\pi/m}} \forall i,j \in \Z$ are symmetries of $\xi_{m-1,1}$. 
Let us fix a choice of the unit normal vector field $N$ on $\xi_{m-1,1}$. Given a great circle $C\subset S^3$, we define the non-exceptional Jacobi field $J_C := K_C \cdot N$, where $K_C$ is defined as in Definition~\ref{l:D:rot} (ii). 

\begin{lem}[Non-exceptional Jacobi fields]\cite[Lemma 5.9]{kapouleas2020index} \label{lem: Jacobi basis}
Let $m \ge 3$.
Then $J_C$, $J_{C^\perp}$, and $J_{C_\phi^{\phi'}}$ for $\phi, \phi' \in \{0, \pi/2\}$ form a basis of the space of non-exceptional Jacobi fields on $\xi_{m-1,1}$. 
\end{lem}

\begin{lem}[Symmetries of Jacobi fields]\cite[Lemma 5.6]{kapouleas2020index} \label{lem: Jacobi symmetry}
Let $m \ge 4$ be an even integer. We define $0_\perp := \pi/2$ and $(\pi/2)_\perp := 0$. Then for $\xi_{m-1,1}$,
\begin{enumerate}[(i)]
\item $J_C$ is odd under $\refl_{\Sigma^0}$ and $\refl_{\Sigma^{\pi/2}}$ and even under $\refl_{\Sigma_0}$ and $\refl_{\Sigma_{\pi/2}}$.
\item $J_{C^\perp}$ is odd under $\refl_{\Sigma_0}$ and $\refl_{\Sigma_{\pi/2}}$ and even under $\refl_{\Sigma^0}$ and $\refl_{\Sigma^{\pi/2}}$.
\item If $\phi, \phi' \in \{0, \pi/2\}$, then $J_{C_\phi^{\phi'}}$ is odd under $\refl_{\Sigma_\phi}$ and $\refl_{\Sigma^{\phi'}}$ and even under $\refl_{\Sigma_{\phi_\perp}}$ and $\refl_{\Sigma^{\phi'_\perp}}$.
\end{enumerate}
\end{lem}

Kapouleas and Wiygul \cite{kapouleas2020index} computed the index and nullity of the Lawson surface $\xi_{m-1,1}$ as follows:

\begin{thm}{\cite[Theorem 6.21]{kapouleas2020index}}\label{thm: Lawson index}
If $m \ge 3$, then the Lawson surface $\xi_{m-1,1}$ has index $2m+2$. Moreover, it has nullity 6 and no exceptional Jacobi fields.
\end{thm}

\subsection{Index and nullity of $\overline{\xi}_{m-1,1}$ in $\R P^3$}
%In what following, we restrict to the case $k=2$ and focus on $\xi_{m-1,1}$. 

In this subsection, we assume that $m \ge 4$ is an even integer. Since $\overline{\xi}_{m-1,1}$ is two-sided, by the discussion at the beginning of Section~\ref{sec: RPn}, we have that with respect to the antipodal map $\tau$, $\Ind(\overline{\xi}_{m-1,1}) = \Ind_\cE({\xi}_{m-1,1})$ and $\Nul(\overline{\xi}_{m-1,1}) = \Nul_\cE({\xi}_{m-1,1})$. Thus our goal is to decompose the index and nullity of $\xi_{m-1,1}$ into even and odd ones with respect to $\tau$. 

Following \cite[Section 6]{kapouleas2020index}, we say a function on a domain in a surface is \emph{piecewise smooth} if it is continuous on the domain, the domain can be subdivided into domains by a finite union of piecewise-smooth embedded curves, and on the closure of each of these domains the function is smooth.
We define
\begin{align*}
V^{\pm\pm}:= \{u\in C_{pw}^\infty(\xi_{m-1,1})  \, : \, u\circ \refl_{\Sigma^0}=\pm u \text{ and } u\circ\refl_{\Sigma^{\pi/2}}=\pm u\,\}, 
\end{align*}
where $C_{pw}^\infty(\xi_{m-1,1})$ denotes the piecewise smooth functions on $\xi_{m-1,1}$ and the $\pm$ signs are taken correspondingly, the first one referring to $\refl_{\Sigma^0}$ and the second one to $\refl_{\Sigma^{\pi/2}}$. 
Since $\refl_{\Sigma^0}$ and $\refl_{\Sigma^{\pi/2}}$ are involutive and commutative, we have the orthogonal decomposition
\begin{align*} 
V = V^{++} \oplusu V^{+-} \oplusu V^{-+} \oplusu V^{--}, 
\end{align*} 
where we use $\oplusu$ to mean that this is not only a direct sum in the sense of linear spaces, but the summands of the direct sum are also invariant under the Jacobi operator $\cL$. Therefore
the same decomposition holds for the corresponding eigenspaces of $\cL$ as well. 
Then similar to the definition of even and odd Morse indices, we let $\Ind(V^{\pm \pm})$ denote the number of negative eigenvalues (counted with multiplicity) of $\cL$ such that the corresponding eigenfunctions fall into $V^{\pm \pm}$. Then 
$$\Ind({\xi}_{m-1,1}) = \Ind(V^{++}) + \Ind(V^{+-}) + \Ind(V^{-+}) + \Ind(V^{--}).$$

Notice that for $u \in C_{pw}^\infty(\xi_{m-1,1})$, the condition $u\circ \refl_{\Sigma^0} = -u$ implies that $u$ vanishes on $\Sigma^{0} \cap \xi_{m-1,1}$, while the condition $u\circ \refl_{\Sigma^0} = u$ implies that if we let $\eta_{i}$ be a choice of unit normal vector field on $\Sigma^{0} \cap \xi_{m-1,1} \subset \xi_{m-1,1}$, then $\nabla_{\eta_{i}} u = 0$ along $\Sigma^{0} \cap \xi_{m-1,1}$. Similarly for the condition $u\circ \refl_{\Sigma^{\pi/2}} = u$. Thus a function in $V^{\pm\pm}$ restricts to a function on $\Omega^{++}_{**}\cap \xi_{m-1,1}$ with Dirichlet or Neumann boundary conditions on the boundary curves $\Sigma^0 \cap \xi_{m-1,1}$ and $\Sigma^{\pi/2} \cap \xi_{m-1,1}$, where $+$ corresponds to the Neumann boundary condition and $-$ corresponds to the Dirichlet boundary condition. 
Conversely, given a piecewise smooth function on $\Omega^{++}_{**} \cap \xi_{m-1,1}$ with Dirichlet or Neumann boundary conditions on the boundary curves, we can use $\refl_{\Sigma^0}$ and $\refl_{\Sigma^{\pi/2}}$ to reflect and get a function in the appropriate $V^{\pm \pm}$.
Thus functions in $V^{\pm\pm}$ are in bijection with piecewise smooth functions on $\Omega^{++}_{**}\cap \xi_{m-1,1}$ with appropriate boundary conditions on the boundary curves $\Sigma^0 \cap \xi_{m-1,1}$ and $\Sigma^{\pi/2} \cap \xi_{m-1,1}$. 

\begin{prop}\label{prop: index decomp}
Let $m \ge 4$ be an even integer. We have the following:
\begin{enumerate}[(i)]
\item \cite[Proposition 6.9]{kapouleas2020index} $\Ind(V^{++}) = 2m-1$.
\item \cite[Proposition 6.13 and Proof of Proposition 6.17]{kapouleas2020index} $\Ind(V^{+-}) = 1$, and %for even $m$, 
the first eigenfunction of $\cL$ in $V^{+-}$ is even with respect to the reflections $\refl_{\Sigma_0}$ and $\refl_{\Sigma_{\pi/2}}$.
\item \cite[Proposition 6.13 and Proof of Proposition 6.17]{kapouleas2020index} $\Ind(V^{-+}) = 1$, and 
%for even $m$, 
the first eigenfunction of $\cL$ in $V^{-+}$ is even with respect to the reflections $\refl_{\Sigma_0}$ and $\refl_{\Sigma_{\pi/2}}$.
\item \cite[Proposition 6.3 (ii)]{kapouleas2020index} $\Ind(V^{--}) = 0$.
\end{enumerate} 
\end{prop}

%From now on we always assume  that $m$ is some even integer greater than or equal to $4$. 

We further study the negative eigenvalues in $V^{++}$. %To simplify notation, let $W: = V^{++}$. 
Using reflections, we can further define spaces
$$V^{++}_{\pm\pm}:= \{u\in V^{++}  \, : \, u\circ \refl_{\Sigma_0}=\pm u \text{ and } u\circ\refl_{\Sigma_{\pi/2}}=\pm u\,\}. 
$$
Be cautious that this is \emph{different} from the spaces $V^{++}_{\pm\pm}$ defined in \cite[(6.4)]{kapouleas2020index}.

Since the four reflections $\refl_{\Sigma^0}$, $\refl_{\Sigma^{\pi/2}}$, $\refl_{\Sigma_0}$ and $\refl_{\Sigma_{\pi/2}}$ are involutive and commutative, we have the orthogonal decomposition
$$V^{++} = V^{++}_{++} \oplusu V^{++}_{+-} \oplusu V^{++}_{-+} \oplusu V^{++}_{--}.$$

Same as before, we let $\Ind(V^{++}_{\pm \pm})$ denote the number of negative eigenvalues (counted with multiplicity) of $\cL$ such that the corresponding eigenfunctions fall into $V^{++}_{\pm \pm}$. Also, the functions in $V^{++}_{\pm \pm}$ are in bijection with piecewise smooth functions on $\Omega^{++}_{++}\cap \xi_{m-1,1}$ with appropriate boundary conditions on the four boundary curves $\Sigma^0 \cap \xi_{m-1,1}$, $\Sigma^{\pi/2} \cap \xi_{m-1,1}$, $\Sigma_0 \cap \xi_{m-1,1}$ and $\Sigma_{\pi/2} \cap \xi_{m-1,1}$, where $+$ corresponds to the Neumann boundary condition and $-$ corresponds to the Dirichlet boundary condition. Then by Proposition~\ref{prop: index decomp} (i), 
\begin{align}
\label{eq: sum}
\Ind(V^{++}_{--}) + \Ind(V^{++}_{+-}) + \Ind(V^{++}_{-+}) + \Ind(V^{++}_{++}) =  \Ind(V^{++}) = 2m - 1.  
\end{align}

We compute these indices:

\begin{prop}\label{prop: index decomp, W}
We have that
\begin{align*}
 \Ind(V^{++}_{--})  = \frac{m}{2} -1, \hspace{1em} \Ind(V^{++}_{+-}) = \Ind(V^{++}_{-+}) = \Ind(V^{++}_{++}) = \frac{m}{2}.
\end{align*}
\end{prop}
\begin{proof}
By symmetry of the surface we have $\Ind(V^{++}_{+-}) = \Ind(V^{++}_{-+})$. Since the functions in $V^{++}_{\pm \pm}$ are in bijection with piecewise smooth functions on $\Omega^{++}_{++}\cap \xi_{m-1,1}$ with appropriate boundary conditions on the four boundary curves, and the $k$-th Neumann eigenvalue of $\cL$ on $\Omega^{++}_{++}\cap \xi_{m-1,1}$ is bounded from above by the $k$-th Dirichlet eigenvalue, we have  
\begin{align}
\label{ineq: chain}
    \Ind(V^{++}_{--}) \le \Ind(V^{++}_{+-}) = \Ind(V^{++}_{-+}) \le \Ind(V^{++}_{++}).
\end{align}

To compute these indices, we define
$$V^{++}_{(\pm)} := \{u\in V^{++} \, : \, \forall i\in\Z \,\, \,\, \, \, u\circ \refl_{\Sigma_{i\pi/m}} = \pm u \}.$$

Note that the space $V^{++}_{(\pm)}$ is denoted $V^{++}_\pm$ in \cite[(6.4)]{kapouleas2020index} but we use the notation $V^{++}_{(+)}$ to emphasize that $V^{++}$ is \emph{not} the direct sum of $V^{++}_{(+)}$ and $V^{++}_{(-)}$.

Recall that the great spheres $\Sigma^0$, $\Sigma^{\pi/2}$ and  $\Sigma_{i\pi/m}$ for $i \in \Z/m\Z$ divide the sphere $S^3$ into $8m$ isometric components given by $\underline{\mathbf{\Omega}} = \{\Omega^j_i\}_{i,j\in\frac{1}{2} + \Z}$. Because of the symmetries of $\xi_{m-1,1}$, these great sphere also divide $\xi_{m-1,1}$ into $8m$ isometric components given by $\{\Omega^j_i \cap \xi_{m-1,1}\}_{i,j\in\frac{1}{2} + \Z}$.

%Be cautious that this time $V^{++}$ is \emph{not} the direct sum of $V^{++}_{(+)}$ and $V^{++}_{(-)}$. 

As before, we let $\Ind(V^{++}_{(\pm)})$ (resp. $\Nul(V^{++}_{(\pm)})$) denote the number of negative eigenvalues (resp. 0-eigenvalues) of $\cL$ such that the corresponding eigenfunctions fall into $V^{++}_{(\pm)}$.
%The functions in $V^{++}_{(\pm)}$ are in bijection with piecewise smooth functions on $( \Omega^{0+}_{i} \cup \Omega^{1-}_{i} )\cap \xi_{m-1,1}$ with appropriate boundary conditions on the four boundary curves $\Sigma^0 \cap \xi_{m-1,1}$, $\Sigma^{\pi/2} \cap \xi_{m-1,1}$, $\Sigma_0 \cap \xi_{m-1,1}$ and $\Sigma_{\pi/m} \cap \xi_{m-1,1}$. 
For any $i \in \Z$, the functions in $V^{++}_{(+)}$ are in bijection with piecewise smooth functions on $\Omega^{\frac{1}{2}}_{i+\frac{1}{2}}\cap \xi_{m-1,1}$ with Neumann boundary conditions on all four boundary curves $\Sigma^0 \cap \xi_{m-1,1}$, $\Sigma^{\pi/2} \cap \xi_{m-1,1}$, $\Sigma_{i\pi/m} \cap \xi_{m-1,1}$ and $\Sigma_{(i+1)\pi/m} \cap \xi_{m-1,1}$, while the functions in $V^{++}_{(-)}$ are in bijection with piecewise smooth functions on $\Omega^{\frac{1}{2}}_{i+\frac{1}{2}}\cap \xi_{m-1,1}$ with Neumann boundary conditions on the boundary curves $\Sigma^0 \cap \xi_{m-1,1}$ and $\Sigma^{\pi/2} \cap \xi_{m-1,1}$, and Dirichlet boundary conditions on the boundary curves $\Sigma_{i\pi/m} \cap \xi_{m-1,1}$ and $\Sigma_{(i+1)\pi/m} \cap \xi_{m-1,1}$.
By Lemma 6.6 and 6.7 of \cite{kapouleas2020index}, we have
$\Ind(V^{++}_{(+)}) = 1$, $\Nul(V^{++}_{(+)}) =0$, $\Ind(V^{++}_{(-)}) = 0$, $\Nul(V^{++}_{(-)}) =1$. 

%Notice that for $u \in W_-$, the condition $u\circ \refl_{\Sigma_{i\pi/m}} = -u$ for all $i$ implies that $u$ must vanish on each $\Sigma_{i\pi/m} \cap \xi_{m-1,1}$. For $u \in W_+$, the condition $u\circ \refl_{\Sigma_{i\pi/m}} = u$ for all $i$ implies that if we let $\eta_{i}$ be a choice of unit normal vector of $\Sigma_{i\pi/m} \cap \xi_{m-1,1} \subset \xi_{m-1,1}$, then $\nabla_{\eta_{i}} u = 0$ along $\Sigma_{i\pi/m} \cap \xi_{m-1,1}$.

Recall from equation~\eqref{Omu++++} that $\Omega^{++}_{++} = \bigcup_{i=1}^{\frac{m}{2}} \Omega_{i+\frac{1}{2}}^{\frac{1}{2}}$.
By Proposition A.1 of \cite{kapouleas2020index}, we have the two inequalities
\begin{align}
\label{ineq: D}
\Ind(V^{++}_{--}) \ge  \Ind(V^{++}_{(-)}) + (\frac{m}{2} -1)\Big(\Ind(V^{++}_{(-)}) + \Nul(V^{++}_{(-)})\Big) = \frac{m}{2} -1
\end{align}
and
\begin{align}
\label{ineq: N}
\Ind(V^{++}_{++}) + \Nul(V^{++}_{++}) \le  \Big(\Ind(V^{++}_{(+)}) + \Nul(V^{++}_{(+)})\Big) + (\frac{m}{2} -1)\Ind(V^{++}_{(+)}) = \frac{m}{2}. 
\end{align}
Combining \eqref{eq: sum}, \eqref{ineq: chain}, \eqref{ineq: D} and \eqref{ineq: N} yields

\begin{align*}
 \Ind(V^{++}_{--})  = \frac{m}{2} -1, \hspace{1em} \Ind(V^{++}_{+-}) = \Ind(V^{++}_{-+}) = \Ind(V^{++}_{++}) = \frac{m}{2}.
\end{align*}
\end{proof}

With these results, we can decompose the index and nullity of $\xi_{m-1,1}$ into even and odd ones as follows:
\begin{thm}
Let $m \ge 4$ be an even integer. With respect to the antipodal map, the Lawson surfaces ${\xi}_{m-1,1}$ in $S^3$ has
\begin{align*}
\Ind_{\cE}({\xi}_{m-1,1}) = m-1,  \hspace{1em} 
&\Ind_{\cO}({\xi}_{m-1,1}) = m+2, \\ 
\Nul_{\cE}({\xi}_{m-1,1}) = 6,  \hspace{1em} 
&\Nul_{\cO}({\xi}_{m-1,1}) = 0.    
\end{align*}
\end{thm}
\begin{proof}
By Theorem~\ref{thm: Lawson index}, the Lawson surface ${\xi}_{m-1,1}$ has index $2m+1$ and nullity 6, and it has no exceptional Jacobi fields.

Notice that the four reflections $\refl_{\Sigma^0}$, $\refl_{\Sigma^{\pi/2}}$, $\refl_{\Sigma_0}$ and $\refl_{\Sigma_{\pi/2}}$ are involutive and commutative, and we can write the antipodal map as $\tau = \refl_{\Sigma^0}\circ \refl_{\Sigma^{\pi/2}}\circ \refl_{\Sigma_0} \circ \refl_{\Sigma_{\pi/2}}$. 

First we look at the nullity. By Lemma~\ref{lem: Jacobi symmetry}, each of the six non-exceptional Jacobi fields $J_C$, $J_{C^\perp}$, and $J_{C_\phi^{\phi'}}$ for $\phi, \phi' \in \{0, \pi/2\}$ are odd with respect to two of the four reflections and even with respect to the other two. Thus all six of them are even with respect to $\tau$. Since by Lemma~\ref{lem: Jacobi basis} these six Jacobi fields form a basis of the space of non-exceptional Jacobi fields on $\xi_{m-1,1}$, we see that all Jacobi fields on $\xi_{m-1,1}$ are even with respect to $\tau$. Thus $\Nul_{\cE}({\xi}_{m-1,1}) = 6$ and $\Nul_{\cO}({\xi}_{m-1,1}) = 0$.

We then look at the index. We decompose the eigenfunctions in each $V^{\pm\pm}$ into odd and even ones.

By Proposition~\ref{prop: index decomp} (iv), $\Ind(V^{--}) = 0$, so $V^{++}$ has no contribution toward either the odd index or the even index. 

By Proposition~\ref{prop: index decomp} (ii) and (iii), $\Ind(V^{+-}) = \Ind(V^{-+}) = 1$ and the first eigenfunctions of $\cL$ in both $V^{+-}$ and $V^{-+}$ are even with respect to the two reflections $\refl_{\Sigma_0}$ and $\refl_{\Sigma_{\pi/2}}$. Thus the first eigenfunctions of $\cL$ in both $V^{+-}$ and $V^{-+}$ are odd with respect to one of the four reflections and even with respect to the remaining three reflections, so they are odd with respect to $\tau$. Thus $V^{+-}$ and $V^{-+}$ each contribute $1$ to the odd index and 0 to the even index. 

For $V^{++}$, notice that functions in $V^{++}_{+-}$ and $V^{++}_{-+}$ are odd with respect to $\tau$ while functions in $V^{++}_{++}$ and $V^{++}_{--}$ are even. Thus by Proposition~\ref{prop: index decomp, W}, the space $V^{++}$ contributes $\Ind(V^{++}_{+-}) + \Ind(V^{++}_{-+}) = m$ to the odd index and contributes $\Ind(V^{++}_{++}) + \Ind(V^{++}_{--}) = m-1$ to the even index. 

Summing these contributions, we find that $\Ind_{\cO}({\xi}_{m-1,1}) = 0+1+1 + m = m+2$ and $\Ind_{\cE}({\xi}_{m-1,1}) = 0+0+0 + (m-1) = m-1$.
\end{proof}
\begin{rem}
Using Lemma~\ref{lem: Jacobi symmetry} and the above theorem, we find that the decomposition of the nullites and indices of the Lawson surfaces ${\xi}_{m-1,1}$ in $S^3$ for $m$ even, $m \ge 4$ is given by
\begin{table}[!ht]
\begin{tabular}{|c|c|c|c|c|c|c|c|c|}\hline
space        &  \multicolumn{2}{c|}{$V^{++}$}                 &  \multicolumn{2}{c|}{$V^{+-}$}             &  \multicolumn{2}{c|}{$V^{-+}$}               &  \multicolumn{2}{c|}{$V^{--}$}                \\\hline
symmetry  &  even &  odd                 & even & odd             & even   & odd               & even &  odd                \\\hline
$\nullity$ & $1$ & $0$ & $2$ & $0$ & $2$ & $0$ & $1$ & $0$\\\hline
$\Index$   & $m-1$ & $m$ & $0$ & $1$ & $0$ & $1$ & $0$ &$0$                  \\\hline
\end{tabular}
\end{table}
\end{rem}

Since $\overline{\xi}_{m-1,1}$ is two-sided, $\Ind(\overline{\xi}_{m-1,1}) = \Ind_\cE({\xi}_{m-1,1})$ and $\Nul(\overline{\xi}_{m-1,1}) = \Nul_\cE({\xi}_{m-1,1})$. We have

\begin{cor}
Let $m \ge 4$ be an even integer. Then the Lawson surface $\overline{\xi}_{m-1,1}$ has index $m-1$ and nullity $6$ in $\R P^3$. 
\end{cor}

Since we know the Clifford torus has index one in $\R P^3$ \cite{do2000compact}, this establishes Theorem~\ref{thm: RP^3}.
%\begin{cor}
%There exist compact embedded two-sided minimal surfaces of every odd index in $\R P^3$.
%\end{cor}

\bibliographystyle{alpha}
\bibliography{main}

\end{document}